\documentclass[11pt]{article}
\usepackage{amssymb}
\usepackage{amsmath}
\usepackage{amsthm}
\usepackage{bm}
\usepackage{algorithm,algpseudocode}
\usepackage{graphicx}
\usepackage[backref]{hyperref} 
\usepackage{subcaption}
\usepackage{svg}
\usepackage[nohead,margin=1.0in]{geometry}
\usepackage[title]{appendix}
\usepackage{verbatim}
\newtheorem{theorem}{Theorem}[section]

\newtheorem{lemma}[theorem]{Lemma}

\newcommand\keywords[1]{\textbf{Keywords}: #1}
\numberwithin{equation}{section}

\title{An Alternating Direction Method of Multipliers for Inverse Lithography Problem}

\author{Junqing Chen\thanks{Department of Mathematical Sciences, Tsinghua University, Beijing 100084, P.R. China.  (\texttt{jqchen@tsinghua.edu.cn}, \texttt{liuhb19@mails.tsinghua.edu.cn}).
} \and Haibo Liu\footnotemark[2]}

\date{March 21, 2023}

\begin{document}

\maketitle

\begin{abstract}
We propose an alternating direction method of multipliers (ADMM) to solve an optimization problem stemming from inverse lithography.  The objective functional of the optimization problem includes three terms: the misfit between the imaging on wafer and the target pattern, the penalty term which ensures the mask is binary and the total variation regularization term. By variable splitting, we introduce an augmented Lagrangian for the original objective functional. In the framework of ADMM method, the optimization problem is divided into several subproblems. Each of the subproblems can be solved efficiently. We give the convergence analysis of the proposed method. Specially, instead of solving the subproblem concerning sigmoid, we solve directly the threshold truncation imaging function which can be solved analytically.  We also provide many numerical examples to illustrate the effectiveness of the method.
\end{abstract}


\section{Introduction}\label{sect1}

Optical lithography system plays a critical role in semiconductor industry.  The optical lithography system can print mask patterns onto the wafer through an optical system.  As shrinkage of integrated circuit device size and subject to the resolution limit of the optical system, the diffraction of mask makes the patterns on wafer distorted very much. To remedy that, many resolution enhancement techniques are proposed and used extensively in optical lithography.  Optical proximity correction(OPC) is the most popular resolution enhancement method. With optical proximity correction, people adjust the mask layout such that the output pattern approximates the desired one. OPC can be divided into two categories: rule based OPC and model based OPC. Rule based method modifies the mask layout and compensates for the warping in local features empirically, which is easy to implement. Model based method relies on the physical and mathematical theories of optical imaging. With model based OPC, the mask pattern is represented by a pixel-wise function or a level-set function, then the OPC process is modeled as an inverse problem. Usually, the inverse problem is formulated as a non-convex optimization problem about the mask pattern. So the model based OPC is usually called inverse lithography technology(ILT) \cite{pang2006inverse}.  The objective functional of this inverse problem is the misfit between the image on wafer and the target pattern. The image on wafer is modeled by threshold and truncation of aerial image projected by the optical system and the threshold is used to mimic the exposure process of photoresist on wafer. By solving the inverse problem, ILT can adjust the shape of mask such that the mask gets close to the optimal one.  It is known that the aerial image of mask is approximated by Abbe formula which can be regarded as convolution \cite{goodman2005}. Besides, the threshold function is discontinuous, then the objective functional of inverse lithography is non-smooth and non-convex, and traditional optimization algorithms are difficult to achieve good results.

Since the importance of the inverse lithography problem, tremendous efforts have been devoted to solve the problem in the past decades. Based on the pixel-wise mask, the inverse lithography problem was stated as nonlinear, constrained minimization problems, and linear, quadratic, and nonlinear formulations of the objective function were considered in \cite{granik2006fast}. The steepest descent methods were proposed in \cite{peng2011gradient, ma2011pixel}. 
A conjugate gradient method was suggested to accelerate the computation in \cite{ma2011pixel}. And there are many other works related to the gradient-based method such as \cite{chan2008inverse,chan2008initialization, ma2007generalized, poonawala2007double, poonawala2008model}. The level set method is another effective way to solve inverse lithography problem. With this method, the target pattern, mask pattern, and wafer pattern are represented by level set functions, then the inverse lithography problem is regarded as a shape and topology optimization problem \cite{shen2009level, shen2019semi}. There are many other miscellaneous methods in inverse lithography, such as the genetic algorithm \cite{erdmann2004toward}, the route of particle swarm optimizer combined with the adaptive nonlinear control strategy \cite{sun2021global}, deep convolution neural network methods \cite{yang2019gan, yang2022generic}. A fuller review of inverse lithography techniques could be found in \cite{pang2021inverse}.

It is known that the inverse lithography problems are ill-posed. Regularization method is an efficient way to remedy the ill-posedness \cite{Engl2000}. The regularization framework was proposed to control the tone and complexity of the synthesized masks\cite{poonawala2007mask}, $L^1$-regularization, wavelet and total-variation(TV) regularizations are discussed in \cite{ma2007generalized}. Sparsity and low-rank regularization terms are used to constrain the solution space and reduce the mask complexity. The split Bregman algorithm is used to solve the inverse optimization problem with sparsity($L^1$ norm) and low rank(nuclear norm) regularizations\cite{ma2019nonlinear}. Specially, the TV regularization is much attractive due to its good edge-preserving property\cite{strong2003edge}. An objective functional that consists of several total-variation regularization terms is introduced in \cite{choy2012robust}, where the total variations about mask and image on wafer are both considered. But the TV regularization is not 
differentiable and it is difficult to implement. Some methods are proposed to deal with the TV-regularization such as smoothness method \cite{chen1999augmented}, dual methods \cite{choy2012robust,chambolle2004algorithm,chan1999nonlinear} and Bregman iteration method\cite{osher2005iterative}.

While various computational schemes have been proposed solving the inverse lithography problems, the computing process of ILT is still complicated. Especially when there are several regularization terms in the loss function \cite{choy2012robust}. Furthermore, since the optimization problem in optical lithography is not convex, the convergence analysis of the solution algorithms is lacking in the literature. The ADMM method \cite{gabay1976dual} has been found to work well in solving non-smooth and non-convex problems, such as phase retrieval problem\cite{wen2012alternating}, compressive sensing \cite{chartrand2013nonconvex}, noisy color image restoration \cite{lai2014splitting}. 
Besides, the convergence analysis of ADMM method has made a lot of progress in recent years \cite{wang2019global}. 

In the present work, we focus on inverse lithography problem with TV regularization and propose an ADMM method to deal with the regularization.
There are several advantages of our method. Firstly, comparing with \cite{choy2012robust}, our method can avoid the tedious computation of total variation of image on wafer and focuses on the TV regularization of mask. 
Secondly, in the framework of ADMM, we split the original optimization problem to three subproblems to avoid the expensive computation of the gradient of a composite function about sigmoid and convolution.
Thirdly,  we establish the convergence of proposed ADMM method under certain conditions. 
Finally, when dealing with the subproblem about the sigmoid function, we use its original nonsmooth form to accelerate the computation.

The paper is organized as follows. In Section \ref{sect2}, we introduce the inverse lithography problem and prove the stability of the solution to inverse problem with total variation regularization. We develop the ADMM method for inverse lithography problem and prove the convergence of the proposed algorithm in Section \ref{sect3}. Some numerical examples are provided in Section \ref{sect4} to show the effectiveness of our algorithm. We draw some conclusion in Section \ref{sect5}.

\section{Inverse lithography problem}\label{sect2}
\subsection{Lithography system and imaging process}
The Figure (\ref{projection_printing}) describes the most important elements inside a photo-lithography system.
\begin{figure}[ht]
	\centering
	\includegraphics[scale=0.3]{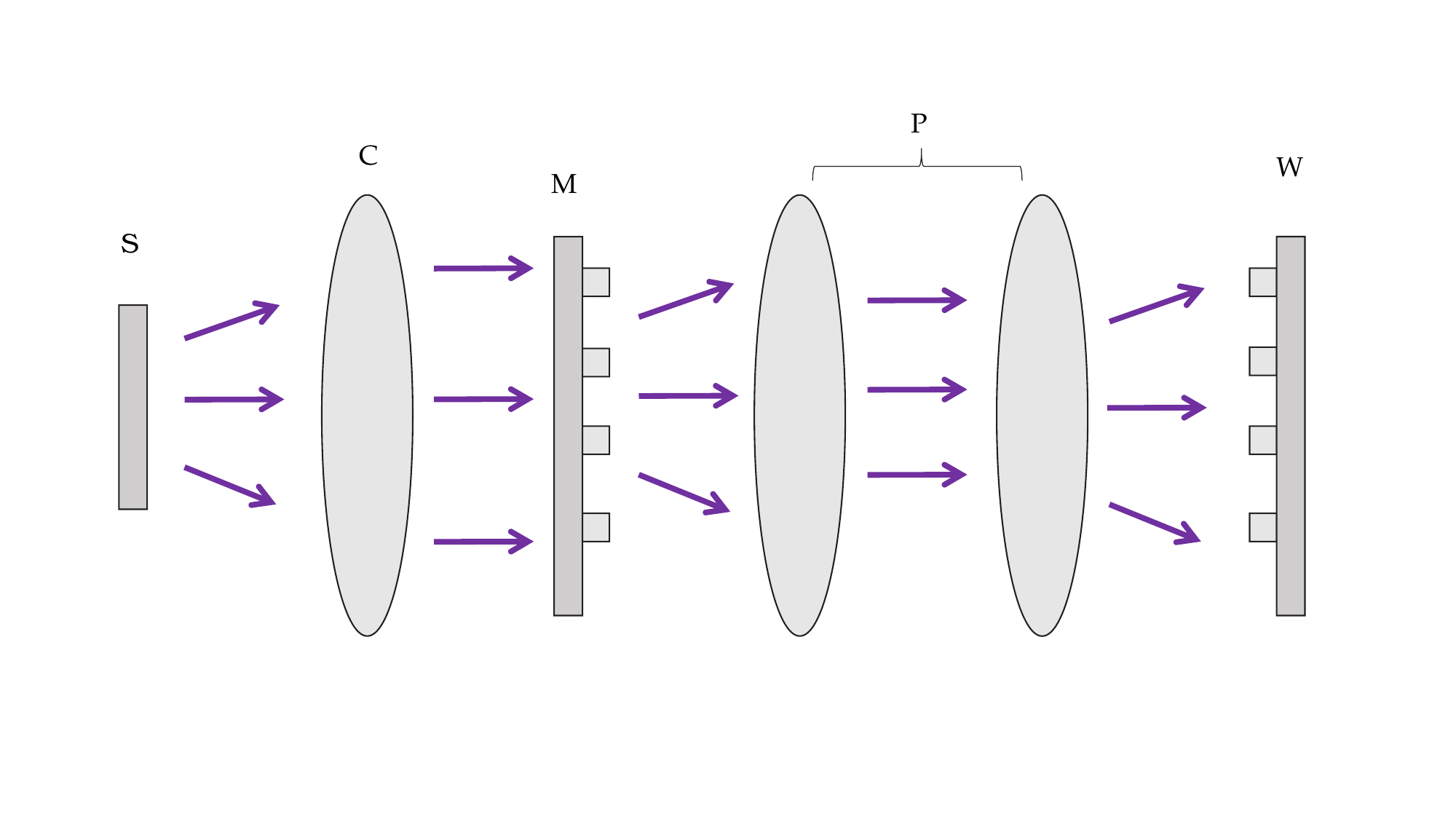}
    \caption{Lithography system}
	\label{projection_printing}
\end{figure}
\noindent The light emitted by the source \{\textbf{S}\} goes through the condenser \{\textbf{C}\}, 
becomes parallel light and then projects the mask pattern \{\textbf{M}\} by the projection lens \{\textbf{P}\} onto the wafer \{\textbf{W}\}. 
With the help of photochemical reaction in photo-resist on wafer, this optical lithography process can transfer the mask pattern to wafer pattern for next stage use.
Under the assumption that the mask is thin, the projection processing can be described by Fourier Optics \cite{goodman2005}. 
Using the Kirchhoff approximation, with the help of Abbe formulation, the aerial image
on the wafer with partially coherent illumination can be approximately described as
\begin{equation*}
	I^a(U)(\bm{r})=\iint_{-\infty}^{\infty}U(\boldsymbol{r_1})U^*(\boldsymbol{r_2})\gamma(\boldsymbol{r_1}-\boldsymbol{r_2})H(\boldsymbol{r}-\boldsymbol{r_1})H^*(\boldsymbol{r}-\boldsymbol{r_2})d\boldsymbol{r_1}d\boldsymbol{r_2},
\end{equation*}
where $\bm{r}=(x,y), \boldsymbol{r_1}=(x_1,y_1 )$, and $\boldsymbol{r_2}=(x_2,y_2)$. $U(\boldsymbol{r})$ is the mask pattern, $\gamma(\boldsymbol{r_1}-\boldsymbol{r_2})$ is the complex degree of coherence, and $H(\boldsymbol{r})$ represents the amplitude impulse response of the optical system. The complex degree of coherence $\gamma(\boldsymbol{r_1}-\boldsymbol{r_2})$ is generally a complex number, whose magnitude represents the extent of optical interaction between two spatial locations $\boldsymbol{r_1}$ and $\boldsymbol{r_2}$ of the light source\cite{wong2001resolution}. 
Specially, for completely coherent system, namely, the light is emitted by a point source, $\gamma(\boldsymbol{r_1}-\boldsymbol{r_2})=1$ and the aerial image can be simplified as
\begin{eqnarray*}
	I_{aerial}(U)(\bm{r})=|H(\textbf{r})*U(\textbf{r})|^2,
\end{eqnarray*}
where * denote the convolution operation. 
As for the point spread function $H(\textbf{r})$ which describes the response of an imaging system to a point source, 
\begin{equation*}
	H(x,y)={F}^{-1}[h(f,g)],
\end{equation*}
where $F$ is the Fourier transform and $h(f,g)$ is the projection lens transfer function. 

In Fourier optics, the projection lens acts as a low pass filter for Fourier transform of the mask. The frequencies which are lower than the projection lens cutoff frequency can pass through the lens and frequencies larger than the cutoff frequency are lost. The cutoff frequency $f_{cut}$ is
\begin{equation*}
	f_{cut}=\frac{NA}{\lambda},
\end{equation*}
where $\lambda$ is the wavelength of the optic wave and $NA$ is the numerical aperture of the projection lens \cite{wong2005optical,ma2011computational}.
The effect of the projection lens is described as a circular transfer function
\begin{equation*}
	h(f,g)= \left\{
	\begin{array}{ll}
		1 & if \sqrt{f^2+g^2}\leq\frac{NA}{\lambda}, \\
		0 & otherwise.
	\end{array}
	\right.
\end{equation*}
Then the corresponding point spread function reads as 
\begin{equation*}
	H(x,y)=\frac{J_1(2\pi r NA/\lambda)}{2\pi r NA/\lambda},
\end{equation*}
and is shown in Figure \ref{bragg}.
\begin{figure}[ht]
	\centering
	\includegraphics[scale=0.4]{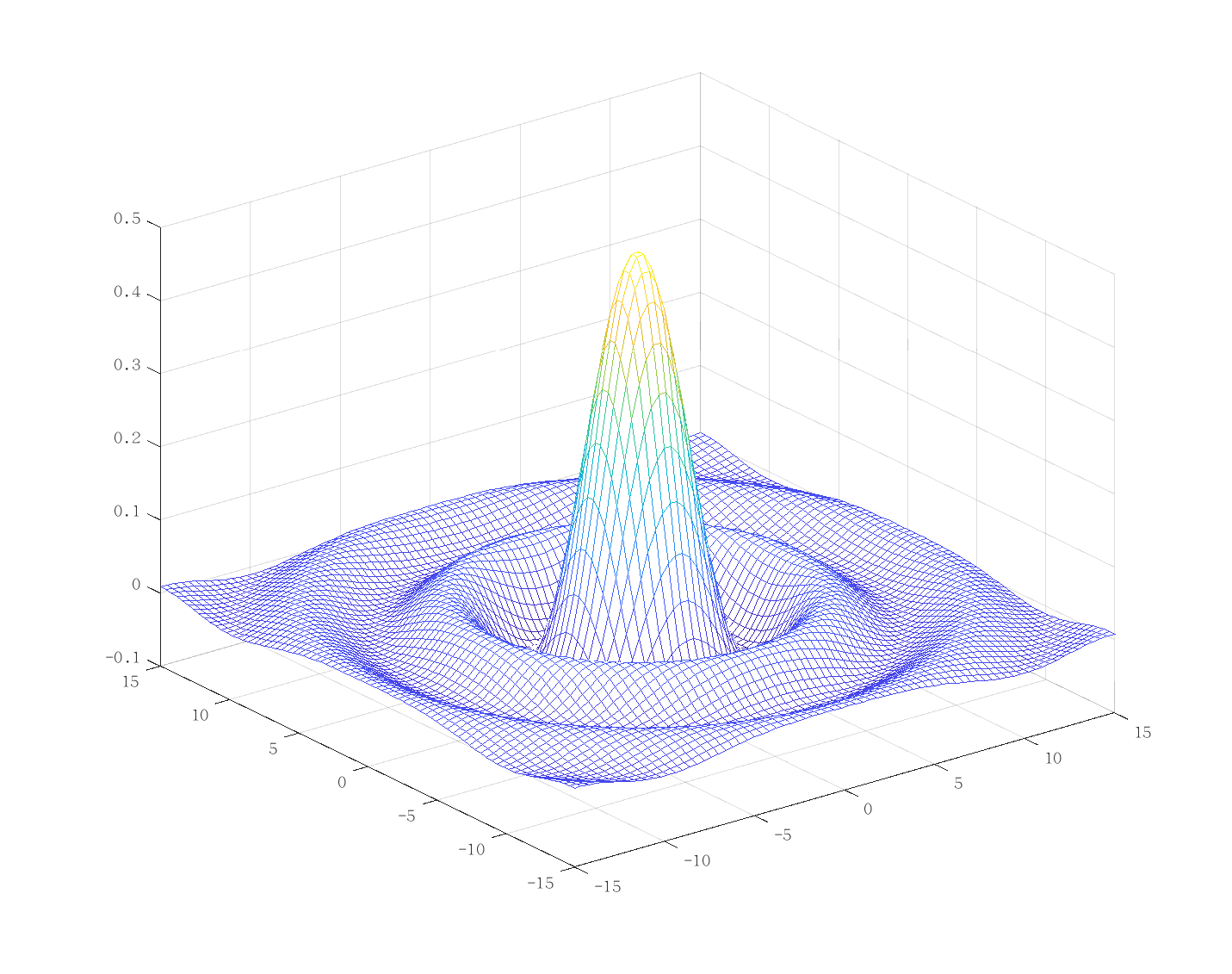}
	\caption{Point spread function for lens with circular aperture.}
	\label{bragg}
\end{figure}
In a real optical lithography system, defocus is an important property that should be considered. It corresponds to a vertical shift of the substrate relative to the plane of best focus. Defocus can be treated similarly with an aberration which describes deviation from the ideal wave front \cite{goodman2005}, even though it is actually not an aberration. Let $W(f,g)$ be an aberration function, the according transfer function for the pupile function is written as 
\[
h(f,g)=
\left\{
\begin{array}{ll}
	e^{-i\frac{2\pi}{\lambda}W(f,g)} & if \quad |(f,g)|\leq \frac{NA}{\lambda}, \\
	0 & otherwise.
\end{array}
\right.
\]
Where the defocus aberration function is 
\begin{equation}\label{defocus}
	W(f,g)=D\sqrt{1-(f^2+g^2)\lambda^2}.
\end{equation}
 and $D$ represents the defocus.

The last step in the simulation of lithography process is the exposure of the photoresist. If we regards the output pattern on wafer is binary, the resist effects could be modeled as a hard threshold operation $T(x)$
\begin{equation}\label{threshd}
	T(x)=\left\{
	\begin{aligned}
		1 & , & x\geq tr, \\
		0 & , & x<tr,
	\end{aligned}
	\right.
\end{equation}
and $tr$ is the threshold.  Then the exposure pattern on wafer is
\begin{eqnarray*}
	I_t(U)(x,y)&=T(I_{aerial}(U)(x,y)).
\end{eqnarray*}
Clearly, $T(x)$ is not a differentiable function. For the optimization method based on gradient, this threshold model is not suitable. Usually, $T$ is approximated with a sigmoid function:

\begin{equation*}
	Sig_a(x)=\frac{1}{1+exp(-a(x-tr))}.
\end{equation*}
For this smoothness, the function value  tends to 1 when $x$ is larger than $tr$, tends to 0 otherwise. The coefficient $a$ controls the steepness. Figure \ref{sigmoid} is the description of a sigmoid function.
\begin{figure}[ht]
	\centering
	\includegraphics[scale=0.6]{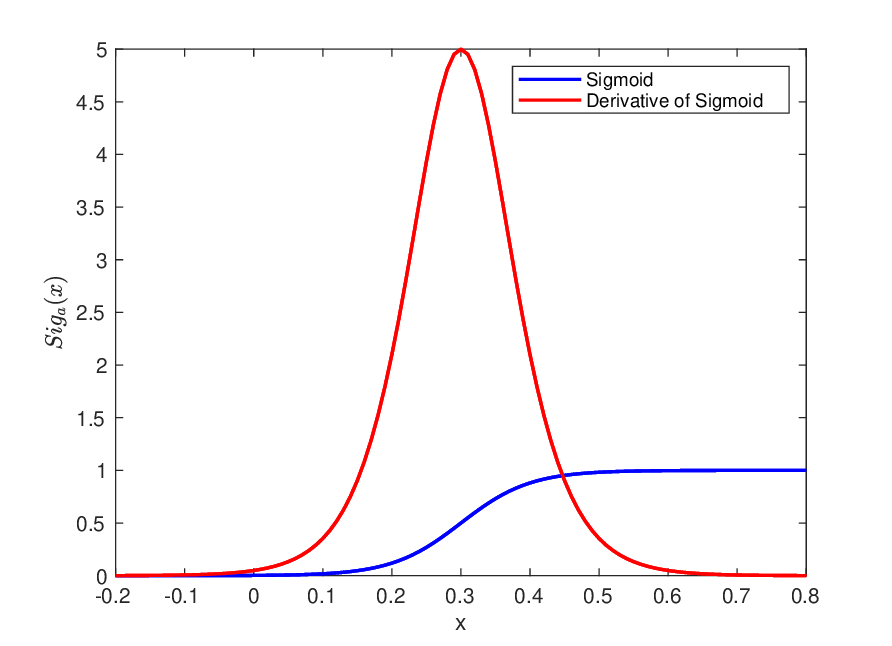}
	\caption{The curves of Sigmoid function and its derivative with $a=20, tr=0.3$.}
	\label{sigmoid}
\end{figure}
Finally, we have the approximate smooth exposure image on wafer
\begin{equation}\label{sigm}
	I(U)(x,y)=Sig_a(|(H* U)(x,y)|^2),
	\nonumber 
\end{equation}
which is the image model we used in the following sections.
\subsection{BV space and Total variation}
Let $\Omega$ be a 2-D bounded Lipschitz domain. In the followings,  we assume that the mask belongs to the BV-space:
$$BV(\Omega)=\{U\in L^1(\Omega):sup \int_{\Omega}{U(x)div\phi(x)dx}<\infty,\phi \in C_c^1(\Omega,R^n),|\phi| \le 1 \},$$
which is equipped with the norm
\begin{equation*}
	\|U\|_{BV}=\|U\|_{L^1(\Omega)}+\|U\|_{TV},
\end{equation*}
where
\begin{equation*}
	\|U\|_{TV}=\int_\Omega |DU|dx\equiv\sup_{\phi\in C^1_c(\Omega),|\phi|\leq 1}\int_\Omega U div \phi dx,
\end{equation*}
is the total variation of $U$. It is known that the BV-space is compactly embedding to $L^1(\Omega)$, which is crucial for the existence and stability of solutions to the inverse problem.  The total variation measures the total image gradient and possesses stronger constraining effects. The regularization based on total variation  has the capability of denoising, deblurring, and yielding sharp edges and is applied in many fields as we have mentioned in Section \ref{sect1}.

\subsection{The inverse lithography problem}
Within the confines of geometrical optics, the image of a point is sharp. But the actual image is smeared because of diffraction. For an imaging system with a circular aperture, the resolution is proportional to
$$\frac{\lambda}{NA},$$
which is always bigger than the critical size of the target pattern on wafer.
So the image on the wafer is blurred. To remedy the blurring and improve the resolution, the resolution enhancement techniques(RETs) have been proposed to minimize mask distortions as they are projected onto wafers. There are many kinds of techniques included in RETs. For more details of RETs technique, we recommend \cite{ma2011computational} and\cite{wong2005optical}.

Among the RETs, inverse lithography technology(ILT) is a promising method and used widely. ILT determines the mask $U(x,y)$ that produces the desired on-wafer results $I(x,y)=I(U)$ by formulating the resolution enhancement process as an inverse problem. This process could be modeled as a optimization problem
\begin{equation*}
	min_{U}||I(U)-I||_{L^2(\Omega)},
\end{equation*}
where U is the unknown mask pattern to be found, I(U) represents the computed field (\ref{sigm}), $I$ is desired lithography pattern.

Over the past years, some of the ideas have been innovatively introduced into ILT and  significantly enhance the quality of reconstructed mask \cite{choy2012robust,li2013efficient,ma2019nonlinear,poonawala2006opc}. Basically, all these methods are based on constructing a cost function comprising one quadratic $L^2$ norm error term and some regularization terms . The regularization terms are used to deal with the ill-posedness of the inverse problem. Then the cost function is expressed as
\begin{equation}\label{origin-reg}
	E(U)=||I(U)-I||^2_{L^2(\Omega)}+\alpha R(U),
\end{equation}
where $\alpha$ is the regularization parameter, and $R(U)$ represents the regularization. 
We assume that the mask is a BV image and binary valued \cite{SIAMB0000033}, which means that the mask $U\in K_0=\{U(x)|U(x)\in BV(\Omega), U(x)=\{0,1\}\}$. Then with \eqref{sigm} and \eqref{origin-reg} in hand, by considering the total variation regularization, the original ILT problem could be modeled as the following constraint optimization problem
\begin{eqnarray}\label{eq1}
	&&\min_{U} ||Sig_a(|H*U|^2)-I||^2_{L^2(\Omega)}+\alpha \|U\|_{TV}\\
	&&\quad\quad s.t. ~U\in K_0. \nonumber
\end{eqnarray}
Since $U(x)$ is binary-valued in $\{0, 1\}$, problem \eqref{eq1} will cause a nonlinear integer programming problem which is NP-hard. Besides, $K_0$ is a non-convex set. So we relax the binary constrain $ U(x)=\{0,1\}$ with $ 0\leq U \leq 1, U(x)\in BV(\Omega)$, add a penalty term $||U(1-U)||_1$.
Then let $K=\{U(x)|U(x)\in BV(\Omega), 0\leq U(x)\leq 1\}$, we have
\begin{eqnarray}\label{objfun}
	&&\min_U   ||Sig_a(|H*U|^2)-I||^2_{L^2(\Omega)}+\beta_1 ||U||_{TV}+\beta_2||U(1-U)||_{L^1(\Omega)}\\
	&&\quad\quad s.t.~U(x)\in K.\nonumber
\end{eqnarray}
It should be mentioned that the cost function \eqref{objfun} is different with the one proposed in \cite{choy2012robust}. Since there is only one regularization term, we can avoid complicated computation of the total variation of aerial image. Besides, there are less parameters needed to be adjusted, so the corresponding algorithm is easy to implement.

\subsection{Stability analysis}
The existence of solution to problem \eqref{objfun} can be proved similarly as in \cite{choy2012robust} where more regularization terms are considered. So we omit the details here. To consider the stability of the solution, we rewrite the objective function as 
$$
E(U;H)=||Sig_a(|H*U|^2)-I||^2_{L^2(\Omega)}+\beta_1||U||_{TV}+\beta_2||U(1-U)||_{L^1(\Omega)}.
$$
The following result shows the stability of solution.
\begin{theorem}
	Assume that $I \in L^\infty (\Omega)$. $H_n \rightarrow H$ in $L^2(\Omega)$  as $n\rightarrow \infty$, $U_n$ is the minimizer of $E(U;H_n)$. Then there exists a subsequence of $\{U_n\}$ which converges to a minimizer of $E(U,H)$ strongly in $L^1(\Omega)$.
\end{theorem}
\begin{proof}
    It is easy to find that
    \begin{eqnarray*}
&&		|H*U|^2(x)-|H_n*U|^2(x)\\
&&\quad		=\int_\Omega (H(y)-H_n(y))U(x-y)dy \cdot \int_\Omega \bar{H}(y)U(x-y)dy\\
&&\quad		+\int_\Omega H_n(y)U(x-y)dy \cdot \int_\Omega (\bar{H}(y)-\bar{H_n}(y))U(x-y)dy.
	\end{eqnarray*}
	Since $H_n\rightarrow H$ in $L^2(\Omega)$, there exists a constant $C$, such that $\|H_n\|_{L^2(\Omega)}\leq C,\|H_n-H\|_{L^2(\Omega)}\leq C$. 
    Be aware of that $0\leq U\leq 1$, we have
    \begin{equation}\label{hu}
		| |H*U|^2(x)-|H_n*U|^2(x) | \le C||H-H_n||_{L^2(\Omega)}, \forall x \in \Omega.
	\end{equation}
	By the definition of $E(U,H)$, we have,
	\begin{eqnarray*}
&&		|E(U;H)-E(U;H_n)|\\
&&\quad		= \int_\Omega |Sig_a(|H*U|^2)-Sig_a(|H_n*U|^2)|\\
&&\quad		\le C ||H-H_n||_{L^2(\Omega)}.
	\end{eqnarray*}
    The last inequality follows from the Lipschitz property of $Sig_a$ and \eqref{hu}, so for any $U\in K$,
	\begin{equation}\label{lim}
		\lim_{n\rightarrow\infty} E(U;H_n)=E(U;H), 
	\end{equation}
	and $E(U;H_n)$ is bounded. Then still by the definition of $E(U,H)$, we have 	
 \begin{equation*}
	    ||U_n||_{TV} \le E(U_n,H_n) \le E(U,H_n),
	\end{equation*}
	and $||U_n||_{TV}$ is bounded.
	From $0\leq U\leq 1$, 	
 \begin{equation*}
		||U_n||_{L^1(\Omega)}=\int_\Omega |U_n| dx \le |\Omega|,
	\end{equation*}
	We can conclude that $\{U_n\}$ is uniformly bounded in $BV(\Omega)$. Since $BV(\Omega)$ is compactly embedded in $L^1(\Omega)$, there exist a subsequence (also denoted by $\{U_n\}$) and a function $U^* \in L^1(\Omega)$ such that	
 \begin{equation}\label{L1}
		U_n\rightarrow U^* \quad strongly \quad in \quad L^1(\Omega).
	\end{equation}
	From Corollary 2.17 of \cite{adams2003sobolev}, there's a subsequence(also denoted by $\{U_n\}$), s.t. 
	
 \begin{equation*}
		U_n\rightarrow U^* \quad a.e.\quad x \in \Omega.
	\end{equation*}
	and a subsequence of $\{H_n\}$ (also denoted by $\{H_n\}$), s.t.	
 \begin{equation*}
		H_n \rightarrow H \quad a.e. \quad x\in \Omega.
	\end{equation*}
	Then	
 \begin{eqnarray*}
		Sig_a(|H_n*U_n|^2)-I  \rightarrow Sig_a(|H*U^*|^2)-I \quad a.e.\quad x \in \Omega\\
		U_n(1-U_n) \rightarrow U^*(1-U^*)\quad a.e.\quad x \in \Omega
	\end{eqnarray*}
	By Fatou’s lemma, we have	
 \begin{eqnarray}\label{pa01}
		&&||Sig_a(|H*U^*|^2)-I||^2_{L^2(\Omega)}+\beta_2||U^*(1-U^*)||_{L^1(\Omega)}\nonumber\\
		&&\quad\le\liminf_{n\rightarrow\infty}||Sig_a(|H_n*U_n|^2)-I||_{L^2(\Omega)}^2+\beta_2||U_n(1-U_n)||_{L^1(\Omega)}.
	\end{eqnarray}
	From \eqref{L1}, we know that \cite{chen1999augmented}	
 \begin{equation}\label{pa02}
		\int_\Omega |DU^*|dx \le \liminf_n \int_\Omega |DU_n|dx.
	\end{equation}
	By \eqref{pa01} and \eqref{pa02}, we have	
 \begin{equation*}
		E(U^*;H)\le \liminf_n E(U_n,H_n).
	\end{equation*}
	From the definition of $U_n$, 	
 \begin{equation*}
		E(U_n;H_n)\le E(U;H_n), \forall U\in K.
	\end{equation*}
	Then	
 \begin{equation*}
		E(U^*;H)\le\liminf_nE(U_n,H_n)\le \liminf_nE(U,H_n)=E(U;H), \forall U\in K,
	\end{equation*}
	hence $U^*$ is a minimizer of $E(U ; H)$.
\end{proof}
We remark that the above result shows the stability with respect to the convolution kernel $H$. Actually, as we have mentioned before, $H$ represents the point spread function of the photo-lithography system which is very complex and many factors can affect $H$. So this stability is important in reality. 
\section{The ADMM method for the inverse lithography problem}\label{sect3}
In this section, we will give an ADMM framework to solve the inverse lithography problem and analyze the convergence, then propose an efficient ADMM method with threshold truncation. 

In the beginning, we give some notations as given in \cite{ma2011computational}. Let the mask domain $\Omega$ be rectangle. We discrete the mask into an $n \times n$ Cartesian grid and let the mask be valued on grid points. In this case, the mask is pixel-based and can be represented by a matrix $U \in R^{n\times n}$. For convenient, we use the vector representation of the matrix $U$ and also denoted $U\in R^{n^2}$. Similarly, we use $I\in R^{n^2}$ represent the desired binary output pattern. A convolution matrix $H \in R^{n^2 \times n^2}$ represent the discrete form of the two dimensional point spread function. Let $D_x, D_y \in R^{n^2} \times R^{n^2}$ are first-order finite difference matrices in the horizontal and vertical directions, respectively.
Let
\begin{equation*}
	DU=(D_x U,D_y U).
\end{equation*}
With these notations, we define 
\begin{equation*}
	||U||_{TV}=\|DU\|_1=||D_xU||_1+||D_yU||_1.
\end{equation*}
By variable-splitting, problem \eqref{objfun} can be recast to
\begin{eqnarray}\label{relax}
	&&\min_{U,V}||Sig_a(|V|^2)-I||_2^2+\beta_1 ||DU||_1+\beta_2||U\odot(1-U)||_1\\
	&& s.t.\quad 0\leq U \leq 1,V=HU.\nonumber
\end{eqnarray}
Here '$\odot$' denotes the element-wise product of matrices.
Then the augmented lagrangian for problem (\ref{relax}) is:
\begin{eqnarray}\label{lag}
	&&L_a(U,V,P) = ||Sig_a(|V|^2)-I||_2^2+\beta_1 ||DU||_1+\beta_2||U\odot(1-U)||_1\nonumber\\
	&&\quad\quad+<P,V-HU>+\frac{\rho}{2}||V-HU||_2^2,
\end{eqnarray}
where $\rho$ is a sufficient large penalty parameter. With this Lagrangian, we have the following alternating direction method of multipliers for given initial $U^0, V^0, P^0$,
\begin{eqnarray}\label{admm}
\left\{\begin{array}{l}
	U^{k+1}=argmin_{U} L_a(U,V^k,P^k),\\
	V^{k+1}=argmin_{V} L_a(U^{k+1},V,P^k),\\
	P^{k+1}=P^{k}+\rho(V^{k+1}-HU^{k+1}).
	\end{array}\right.
\end{eqnarray}
\subsection{Convergence analysis}
The convergence analysis of \eqref{admm} is inspired by the global convergence result of \cite{wang2019global}. In order to make the proof more concise, we first give some notations. We define the following functions

\begin{equation*}
	h_a(V)=||Sig_a(|V|^2)-I||_2^2,
\end{equation*}
then the derivative with respect to $V$
\begin{equation*}
	\nabla h_a(V)=2aV\odot (Sig_a(|V|^2)-I)\odot(1-Sig_a(|V|^2))\odot Sig_a(|V|^2).
\end{equation*}
\begin{figure}[ht]
    \centering
	\includegraphics[scale=0.45]{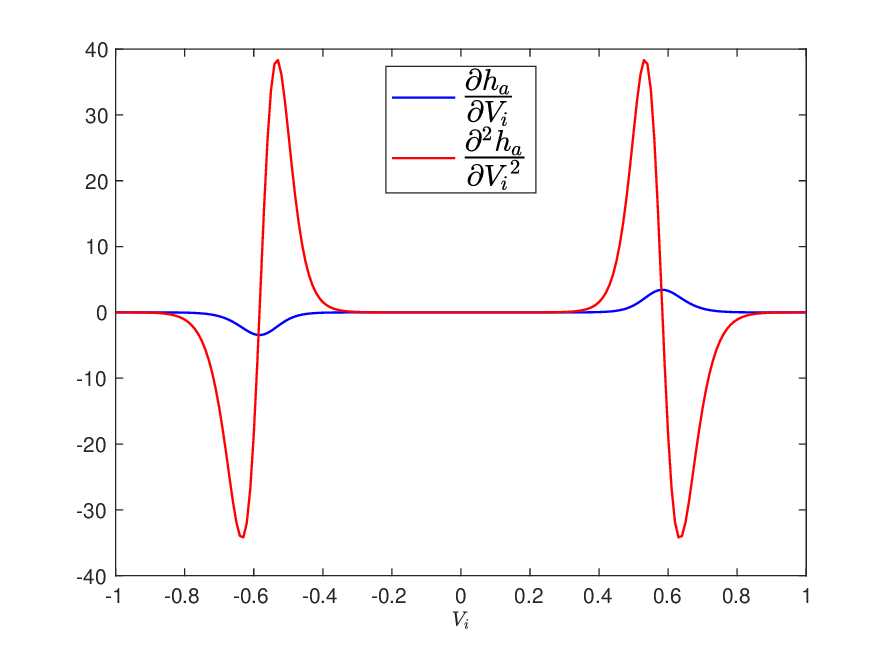}
	\includegraphics[scale=0.45]{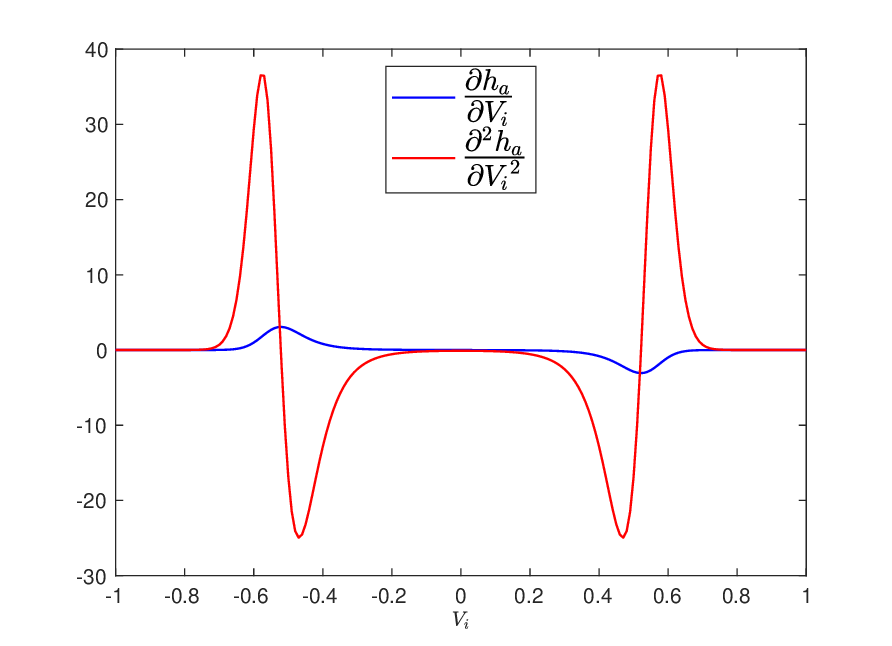}
	\caption{The $i$-th entry of $\nabla h_a(V)$ and $i$-th diagonal entry of Hessian of $h_a(V)$ with $a=20, tr=0.3$. The left is with $I_i=0$ and the right is with $I_i=1$ .}
	\label{ha}
\end{figure}
Clearly, $\nabla h_a(V)$ has the same size with $V$, the Hessian of $h_a(V)$ is a matrix. According to the definition of $h_a(V)$, each entry of $\nabla h_a(V)$ is a single value function, then the hessian of $h_a(V)$ is diagonal. In Figure \ref{ha}, we plot the $i$-th entry of $\nabla h_a(V)$ and the according diagonal entry of Hessian  with $a=20,tr=0.3$. It is easy to know that $\nabla h_a(V)$ is Lipschitz continuous for any given $a$. We suppose the Lipschitz constant is $L_h$. We define 
\begin{equation*}
	f(U)=\beta_1||DU||_1+\beta_2||U\odot(1-U)||_1.
\end{equation*}
In such situation, the Lagrangian function can be recast as
\begin{equation}\label{La}
	L_a(U,V,P) = f(U)+h_a(V)+<P,V-HU>+\frac{\rho}{2}||V-HU||_2^2.
\end{equation}
For convenience, we omit the subscript $a$ in the followings. Then
\begin{eqnarray}\label{approx}
\left\{\begin{array}{l}
	U^{k+1}=argmin_{U} L(U,V^k,P^k),  \\
	V^{k+1}=argmin_{V} L(U^{k+1},V,P^k),  \\
	P^{k+1}=P^{k}+\rho(V^{k+1}-HU^{k+1}).
	\end{array}\right.
\end{eqnarray}
Now we will state and prove the main theorem about the convergence.
\begin{theorem}\label{thm}
	Suppose sequence $(U^k, V^k, P^k)$ is generated by \eqref{approx}, assume that $\frac{\rho}{2}-\frac{L_h}{\rho}-L_h>0$ , then there is a subsequence which converges to a stationary point of $L(U,V,P)$.
\end{theorem}
We begin our analysis by showing a sufficient decrease property of the Lagrange function.
\begin{lemma} \label{lemma1}
	If $\frac{\rho}{2}-\frac{L_h}{\rho}-L_h>0$, where $L_h$ is the Lipschitz constant of $\nabla h$, there is a constant $C=C(\rho,L_h) > 0$ such that for sufficiently large k, we have
	\begin{equation}\label{decrease}
		L( U^k,V^k, P^k)-L(U^{k+1},V^{k+1},P^{k+1})\ge C||V^{k+1}-V^{k}||_2^2.
	\end{equation}
\end{lemma}
\begin{proof}
	With the optimality condition of $V^{k}$	
 \begin{equation*}
	    \nabla h(V^k)+P^{k-1}+\rho (V^k-HU^k)=0,
	\end{equation*}
	and               	
 \begin{equation}
		P^k = P^{k-1} +\rho(V^k-HU^k),\label{Pk}
	\end{equation}
	we have
\begin{equation}
		P^k = -\nabla h(V^k).\label{Pk1}
	\end{equation}
	Then by the Lipschitz continuity of $\nabla h$,
\begin{equation*}
		||P^{k+1}-P^k||_2 = ||\nabla h(V^k)-\nabla h(V^{k+1})||_2 \le L_h||V^k-V^{k+1}||_2.
	\end{equation*}
	With the definition of $U^{k+1}$, it follows that
\begin{equation}
		L(U^k,V^k,P^k)-L(U^{k+1},V^k,P^k)\ge 0,\label{de1}
	\end{equation}
	and
\begin{eqnarray*}
	&&L(U^{k+1}, V^k,P^k)-L(U^{k+1},V^{k+1},P^{k+1})\\
		&&\quad=h(V^k)-h(V^{k+1})-<P^{k+1},V^{k+1}-HU^{k+1}>\\
		&&\quad+<P^{k},V^{k}-HU^{k+1}>+\frac{\rho}{2}||V^{k}-HU^{k+1}||_2^2-\frac{\rho}{2}||V^{k+1}-HU^{k+1}||_2^2\\
		&&\quad=h(V^k)-h(V^{k+1})+<P^{k}-P^{k+1},V^{k}
		-HU^{k+1}>\\
		&&\quad-<P^{k+1},V^{k+1}-V^k>+\frac{\rho}{2}||V^{k}-V^{k+1}||_2^2\\
		&&\quad+\rho<V^k-V^{k+1},V^{k+1}-HU^{k+1}>\\
		&&\quad=h(V^k)-h(V^{k+1})+<P^{k+1},V^{k}-V^{k+1}>+\frac{\rho}{2}||V^{k}-V^{k+1}||_2^2\\
		&&\quad+<P^{k}-P^{k+1},V^{k}-V^{k+1}>+<P^{k}-P^{k+1},V^{k+1}-HU^{k+1}>\\
		&&\quad+\rho<V^k-V^{k+1},V^{k+1}-HU^{k+1}>
		\end{eqnarray*}
	Then combining \eqref{Pk} and \eqref{Pk1} and using Lipschitz property of $\nabla h(V)$,  we have
\begin{eqnarray*}
	&& L(U^{k+1}, V^k,P^k)-L(U^{k+1},V^{k+1},P^{k+1})\\
		&&\quad=h(V^k)-h(V^{k+1})-<\nabla h(V^{k+1}),V^{k}-V^{k+1}>\\
		&&\quad+\frac{\rho}{2}||V^{k+1}-V^{k}||_2^2-\frac{1}{\rho}||P^{k+1}-P^k||_2^2\\
		&&\quad\geq-L_h(||V^{k+1}-V^k||_2^2)+\frac{\rho}{2}||V^{k+1}-V^{k}||_2^2-\frac{L_h}{\rho}||V^{k+1}-V^{k}||_2^2\\
		&&\quad\geq(\frac{\rho}{2}-\frac{L_h}{\rho}-L_h)(||V^{k+1}-V^k||_2^2).
	\end{eqnarray*}
	Finally, with the help of \eqref{de1} and the above inequality, we obtain the result (\ref{decrease}).
\end{proof}
The next lemma tells us that the Lagrangian is bounded below and the sequence $(U^k, V^k, V^k)$ is bounded.
\begin{lemma} \label{bound}
	Under the same condition as Lemma \ref{lemma1},  $L(U^k,V^k,P^k)$ is lower bounded and $(U^k,V^k,P^k)$ is bounded.
\end{lemma}
\begin{proof}
Since $\rho>2L_h$, by the definition of $L(U,V,P)$, we have	
 \begin{eqnarray*}
		&& L(U^k,V^k,P^k)=f(U^k)+h(V^k)+<P^k,V^k-HU^k>+\frac{\rho}{2}||V^k-HU^k||_2^2\\
		&&\quad=f(U^k)+h(V^k)+<P^k,V^k-V'>+\frac{\rho}{2}||V^k-HU^k||_2^2\\
		&&\quad=f(U^k)+h(V^k)-<\nabla
		h(V^k),V^k-V'>+\frac{\rho}{2}||V^k-HU^k||_2^2\\
		&&\quad\ge f(U^k)+h(V')+(\frac{\rho}{2}-L_h)||V^k-HU^k||_2^2\\
		&&\quad\ge \frac{L_h}{\rho}||V^k-HU^k||_2^2\ge 0.
	\end{eqnarray*}
	Lemma \ref{lemma1} tells us that  $L(U^k,V^k,P^k)$ is upper bounded, then $||V^k-HU^k||$ is bounded. We know that $0\le {U^k} \le 1$, then ${HU^k}$ is bounded. So $V^k$ is bounded and $P^k = -\nabla h(V^k)$ is bounded, too.
\end{proof}
\begin{lemma}\label{VP}
The sequence $L(U^k,V^k,P^k)$ is convergent and 
$$\lim_{k\rightarrow\infty}||V^{k}-V^{k+1}||_2\rightarrow 0,\lim_{k\rightarrow\infty}||P^{k}-P^{k+1}||_2\rightarrow 0.$$
\end{lemma}
\begin{proof}
Based on Lemma \ref{lemma1} and Lemma \ref{bound}, we know that the sequence 
$\{L(U^k,V^k,P^k)\}$ is  convergent, then by Lemma \ref{lemma1}, 
we have $||V^{k}-V^{k+1}||_2\rightarrow 0(k \rightarrow 0)$.
Since $P^k=-\nabla h(V^k)$, we have 
$$P^k-P^{k+1}=-\nabla h(V^k)+\nabla h(V^{k+1}).$$
Because $\nabla h(V)$ is Lipschitz continuous,  we can get 
$$\lim_{k\rightarrow\infty}||P^{k}-P^{k+1}||_2\rightarrow 0.$$
\end{proof}
 We now establish the following two convergence results:
\begin{lemma}\label{d}
	For any k, there is a $\bar{d}^k \in  \partial_U L(U^{k+1},V^{k+1},P^{k+1})$ such that 	
 \begin{equation*}
		||\bar{d}^k||_2 \rightarrow 0 \mbox{ as } k\rightarrow \infty.
	\end{equation*}
\end{lemma}
\begin{proof}
	By \eqref{La}, we have	
 \begin{equation*}
		\partial_U L(U^{k+1},V^{k+1},P^{k+1})=\partial f(U^{k+1})-H^T P^{k+1}-\rho H^T(V^{k+1}-HU^{k+1}).
	\end{equation*}
	From the optimality condition of $U^{k+1}$	
 \begin{equation*}
		\partial f(U^{k+1}) -H^T P^k -\rho H^T(V^{k} -HU^{k+1}) =0.
	\end{equation*}
	Let	
 \begin{equation*}
		d^k=H^T P^k +\rho H^T(V^{k} -HU^{k+1})\in \partial f(U^{k+1}). 
	\end{equation*}
	Then we have	
 \begin{eqnarray*}
		&&\bar{d}^k=d^k-H^T P^{k+1}-\rho H^T(V^{k+1}-HU^{k+1})\\
		&&\quad=H^T(P^{k}-P^{k+1})+\rho H^T(V^{k}-V^{k+1})
		\in \partial_U L(U^{k+1},V^{k+1},P^{k+1}),
	\end{eqnarray*}
	and by virtual of Lemma \eqref{lemma1} and \eqref{bound}, 	
 \begin{equation*}
		||\bar{d}^k||_2 \le C'(||P^{k}-P^{k+1}||_2+||V^{k}-V^{k+1}||_2)\le C||V^{k}-V^{k+1}||_2 \rightarrow 0.
	\end{equation*}
\end{proof}
\begin{lemma}\label{limit}
	If $(U^*, V^*, P^*)$ is the limit point of a subsequence $(U^{k_s}, V^{k_s}, P^{k_s})$, then
	$L(U^*, V^*, P^*) = \lim_{s\rightarrow \infty} L(U^{k_s}, V^{k_s}, P^{k_s})$.
\end{lemma}
\begin{proof}
    We first note that $L(U^{k_s}, V^{k_s}, P^{k_s})$ is monotonic nonincreasing, which implies the convergence of $L(U^{k_s}, V^{k_s}, P^{k_s}).$ Since $L$ is lower semicontinuous (l.s.c.), we have	
 \begin{equation}\label{luvp1}
		\lim_{s\rightarrow \infty} L(U^{k_s}, V^{k_s}, P^{k_s}) \ge L(U^*, V^*, P^*).
	\end{equation}
	Because the only potentially discontinuous terms in $L$ is $f(U)$ , we have
	
 \begin{equation}
		\lim_{s\rightarrow \infty} L(U^{k_s}, V^{k_s}, P^{k_s})-L(U^*, V^*, P^*) \le \limsup_{s\rightarrow \infty} f(U^{k_s})-f(U^*) .\label{luvp2}
	\end{equation}
 However, because $U^{k_s}$ is the optimal solution for the subproblem	
 \begin{equation*}
		\min_{U} L(U,V^{k_s-1},P^{k_s-1}),
	\end{equation*}
	we know $L(U^{k_s},V^{k_s-1},P^{k_s-1}) \le L(U^*,V^{k_s-1},P^{k_s-1})$. Since the only discontinuous part of $L(U,V,P)$ is $f(U)$, we have
$$\limsup_{s\rightarrow \infty}L(U^{k_s},V^{k_s-1},P^{k_s-1})- L(U^*,V^{k_s-1},P^{k_s-1})=\limsup_{s\rightarrow \infty}  f(U^{k_s})-f(U^*) ,$$
	then $\limsup_{s\rightarrow \infty}  f(U^{k_s})-f(U^*) \le 0$. Hence,  by virtual of \eqref{luvp1} and \eqref{luvp2},
	
 \begin{equation*}
		\lim_{s\rightarrow \infty} L(U^{k_s}, V^{k_s}, P^{k_s})-L(U^*, V^*, P^*) =0.
	\end{equation*}
\end{proof}
Finally, we establish the proof of theorem \ref{thm}:
\begin{proof}
By Lemma \ref{bound}, the sequence $(U^k,V^k,P^k)$ is bounded,
so there exists a convergent subsequence denoted by $(U^{k_s},V^{k_s}, P^{k_s})$ and
a limit point $(U^*, V^*, P^*)$, such that $(U^{k_s},V^{k_s}, P^{k_s}) \rightarrow (U^*, V^*, P^*)$ as $s\rightarrow +\infty$.
By Lemma \ref{lemma1} and \ref{bound}, $L(U^{k_s},V^{k_s},P^{k_s})$ is monotonically nonincreasing and lower bounded, 
and therefore $||V^{k_s} - V^{k_{s+1}}|| \rightarrow 0$ as $s \rightarrow 0$. Based on \ref{d}, 
there exists $\bar{d}^s\in\partial_U L(U^{k_s},V^{k_s},P^{k_s})$ such that $||\bar{d}^s|| \rightarrow 0$.
Based on Lemma \ref{limit}, $L(U^*,V^*, P^*) = lim_{s\rightarrow \infty} L(U^{k_s}, V^{k_s}, P^{k_s})$.
By definition of general subgradient, we have	
 \begin{equation*}
		L(U,V^{k_s},P^{k_s})-L(U^{k_s},V^{k_s},P^{k_s})-<\bar{d}^s,U-U^{k_s}> \ge 0, \quad \forall U.
	\end{equation*}
Taking limit of $s$, since $L(U,V,P)$ is continuous with respect to $V$ and $P$, we have	
 \begin{equation*}
		L(U,V^*,P^*)-L(U^*,V^*,P^*) \ge 0, \quad \forall U.
	\end{equation*}
Therefore, $0 \in \partial_U L(U^*,V^*, P^*)$.\\
By the optimality of $V^{k_s}$, we know that 
\begin{eqnarray*}
0=\partial_{V}L(U^{k_s},V^{k_s},P^{k_s-1})=\nabla h(V^{k_s})+P^{k_s-1}+\rho(V^{k_s}-HU^{k_s})	\end{eqnarray*}
Clearly, $\partial_V L(U,V,P)$ is continuous with respect to $U, V, P$, then by Lemma \ref{VP}, we have
\begin{eqnarray*}
\partial_{V}L(U^*,V^*,P^*)=\nabla h(V^*)+P^*+\rho(V^*-HU^*)\\
=\lim_{k\rightarrow\infty}\nabla h(V^{k_s})+P^{k_s-1}+\rho(V^{k_s}-HU^{k_s})=0 .
\end{eqnarray*}
Still by the optimality of $V^{k_s}$ ,
$$\rho(V^{k_s}-HU^{k_s})=-\nabla h(V^{k_s})-P^{k_s-1}=\nabla h(V^{k_s-1})-\nabla h(V^{k_s}),$$
then we have	
 $$\partial_P L(U^*,V^*, P^*)=V^*-HU^*=\lim_{s\rightarrow\infty}V^{k_s}-HU^{k_s}=0.$$
Here we use the Lipschitz property of $\nabla h$ and Lemma \ref{VP}.
So $0 \in \partial L(U^*,V^*, P^*)$. 
\end{proof}

\subsection{The ADMM algorithm with threshold truncation}\label{sect32}
\subsubsection{U-subproblem}
For given $V^{k}, P^k$, the U-update step in \eqref{admm} reduces to solving:
\begin{equation}\label{Usub}
	U^{k+1}=argmin_U||HU-W||_2^2+ \beta_1||DU||_1+\beta_2||U\odot(1-U)||_1,
\end{equation}
subject to $0\le U \le 1$, where
\begin{equation}\label{eq2}
	W=V^{k}+\frac{1}{\rho} P^k.
\end{equation}
Let
\begin{equation*}
	\Phi(U)=(\beta_1DU,\beta_2U\odot(1-U)),
\end{equation*}
we can solve problem \eqref{Usub} using split Bregman iteration method \cite{goldstein2009split}:
\begin{eqnarray}\label{breg}
\left\{\begin{array}{l}
	U^{k+1}=\arg\min_U ||HU-W||_2^2+ \frac{\gamma}{2}||d^k-\Phi(U)-b^k||_2^2,\\
	d^{k+1}=\arg\min_d  ||d||_1 + \frac{\gamma}{2}||d-\Phi(U^{k+1})-b^k||_2^2,\\
	b^{k+1}=b^{k}+\Phi(U^{k+1})-d^{k+1}.
	\end{array}\right.
\end{eqnarray}
Denote
\begin{eqnarray*}
    &&F(U)=||HU-W||_2^2+\frac{\gamma}{2}||d-\Phi(U)-b||_2^2\nonumber\\
	&&\quad =||HU-W||_2^2+\frac{\gamma}{2}(||d_1-\beta_1DU-b_1||_2^2+||d_2-\beta_2 U\odot(1-U)-b_2||_2^2),
\end{eqnarray*}
then by a direct calculation
\begin{eqnarray}\label{gradient}
	&&\nabla F(U) = 2Re\{H^*(HU-W)\}-\gamma\beta_1 D^T(d_1-\beta_1DU-b_1)\nonumber\\
	&&\quad+\gamma\beta_2(d_2-\beta_2 U\odot(1-U)-b_2)\odot(2U-1).
\end{eqnarray}
The first problem is solved with gradient decent method. The initial value is $\tilde{U}^0=U^k$, the iteration is 
\begin{equation}
	\tilde{U}^{m+1}=\tilde{U}^m-\eta_m \nabla F(\tilde{U}^m), m=0,1,...\label{gda}
\end{equation}
Here the coefficient $\eta_m$ is given by the Armijo line search method. Armijo rule allows for the step-size to be computed using only local properties of the objective, as opposed to other approaches that use global quantities like the gradient’s Lipschitz constant. The selection process of step size works as following. Fix parametes $0 < \alpha < 0.5, 0 < \beta < 1$, for the objective function $f$, start with an initial step size $t = t_0$, when
\begin{equation}\label{armijo1}
	f(x-t\nabla f(x))>f(x)-\alpha t||\nabla f(x)||^2,
\end{equation}
we update step size as $t = \beta t$. We repeat this process until \eqref{armijo1} fails, then choose $t$ as the suitable step size. When the convergence is achieved at $m=M$, we let $U^{k+1}=\tilde{U}^M$. To make the mask $U$ satisfy  $0\le U\le1$, we project $U$ onto $[0,1]$ by
\begin{equation*}
	U=\min\{\max\{0,U\},1\}.
\end{equation*}
Finally, we update $d$ using the shrinkage method,
\begin{equation}\label{pl}
	d^{k+1}=shrink(\Phi(U^{k+1})+b^k,1/\gamma).
\end{equation}
\subsubsection{V-subproblem}
For given $U^{k+1},P^k$, after some calculations, the V-update step becomes to the following problem,
\begin{equation}\label{Vsub}
	V^{k+1}=\arg\min_{V} ||Sig_a(|V|^2)-I||_2^2+\frac{\rho}{2}||V-W^k||_2^2, 
\end{equation}
where
\begin{equation*}
	W^k=HU^{k+1}-\frac{1}{\rho} P^k.
\end{equation*}
Fortunately, the V-subproblem could be decomposed to a series of single variable optimization problem:
\begin{equation}\label{single}
	V_i^{k+1}=\arg\min_{V_i} (Sig_a(|V_i|^2)-I_i)^2+\frac{\rho}{2}(V_i-W_i^{k})^2.
\end{equation}
As we have mentioned before, $Sig_a$ is a smoothness of the truncation function. Then, to avoid the nonlinear iteration to solve \eqref{single},  we solve the following problem directly,
\begin{equation}\label{singleT}
	V_i^{k+1}=\arg\min_{V_i} (T(|V_i|^2)-I_i)^2+\frac{\rho}{2}(V_i-W_i^{k})^2.
\end{equation}
Where $T$ is the operator defined in \eqref{threshd}.
Problem \eqref{singleT} has a closed-form solution:
\begin{equation}\label{V}
V^{k+1}_i=
\left\{
\begin{array}{ll}
	W^{k}_i&,\quad |W_i^k|<\sqrt{tr} \quad and \quad I_{i}=0,\\
	W^{k}_i&,\quad |W_i^k|<\sqrt{tr} \quad and \quad I_{i}=1  \quad and \quad \frac{\rho}{2}|W_{i}^{k}|-\sqrt{tr}|^2>1,\\
	W^{k}_i&,\quad |W_i^k|>\sqrt{tr} \quad and \quad I_{i}=0  \quad and \quad \frac{\rho}{2}|W_{i}^{k}|-\sqrt{tr}|^2>1,\\
	W^{k}_i&,\quad |W_i^k|>\sqrt{tr} \quad and \quad I_{i}=1,\\
	\sqrt{tr} & \cdot sgn(W^{k}_i), \quad else.
\end{array}
\right.
\end{equation}
Figure \ref{subv} shows the curves of objective function in \eqref{singleT} with $W_i^k=0.2, \rho=1,tr=0.3$ and different $I_i$. We can see that $V_i^{k+1}=\sqrt{tr} sgn(W^k_i)=\sqrt{0.3}$ and $V_i^{k+1}=W^{k}_i$ take the minimums for $I_i=1$ and $I_i=0$ respectively, which confirms the formula \eqref{V}.
\begin{figure}[htbp]
	\centering
	\includegraphics[scale=0.45]{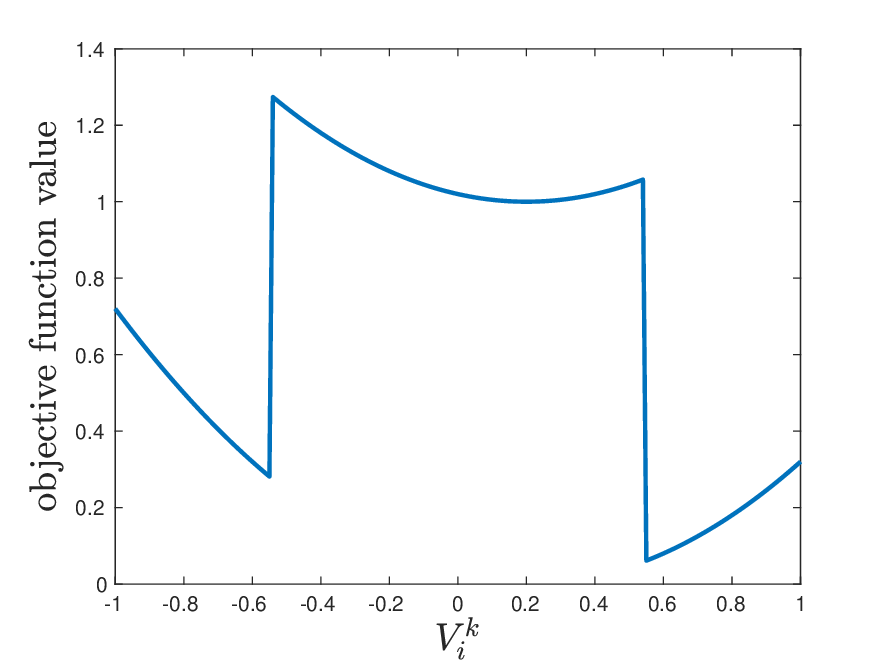}
	\includegraphics[scale=0.45]{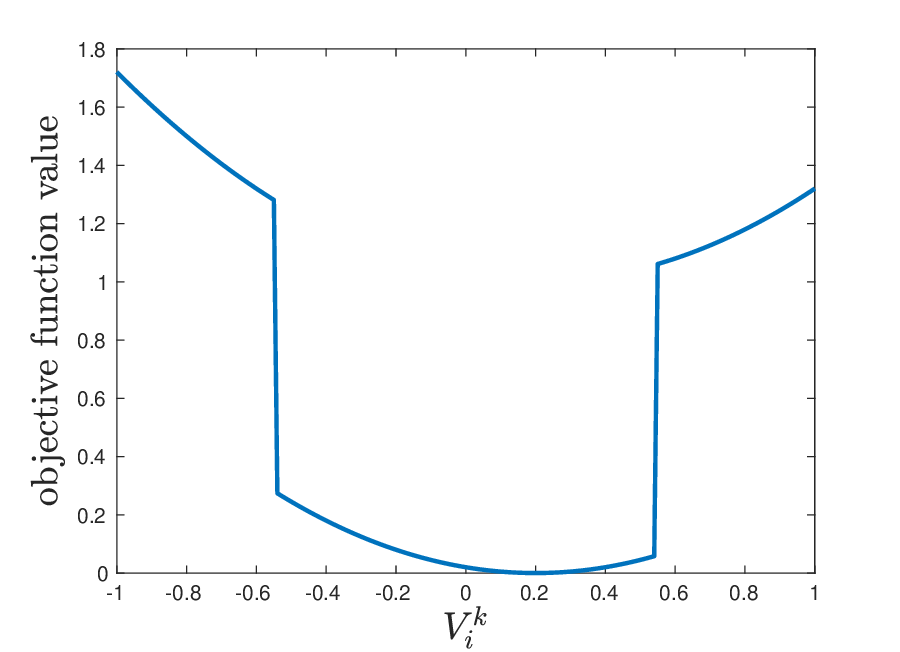}
	\caption{The curves of objective function with $W_i^k=0.2, \rho=1,tr=0.3$.  The left is with $I_i=1$, the right is with $I_i=0$. }
	\label{subv}
\end{figure}
Now, as summary, we give the detailed Algorithm \ref{alg-admm} for minimization problem (\ref{relax}).
\begin{algorithm}[t]
	\caption{The ADMM with threshold function approximation}
	\label{alg-admm}
	\hspace*{0.02in} {\bf Input:} desired image $I$\\
	\hspace*{0.02in} {\bf Output:} optimized mask $U$\\
	\hspace*{0.02in} {\bf Parameters:}$\beta_1,\beta_2,\rho,\gamma,\eta$\\
	\hspace*{0.02in} {\bf Initialize:} $U=I,V=H*U,P=1$
	\begin{algorithmic}
		\While{error$>$tol}
		\State compute $W=V+P/\rho$
		\State Initialize $d=\Phi(U),b=0$
		\While{error$>$tol}
		\State update U with gradient decent method \eqref{gda}
		\State project $U$ onto $[0,1]$:
		\begin{equation*}
			U=min\{max\{0,U\},1\}
		\end{equation*}
		\State update $d=shrink(b+\Phi(U),\frac{1}{\gamma})$
		\State update $b=b+\Phi(U)-d$
		\EndWhile
		\State update $V$ with (\ref{V})
		\State update $P$ with $P=P+\rho(V-HU)$
		\EndWhile\\
	    \Return $U$
	\end{algorithmic}
\end{algorithm}

\section{Numerical examples}\label{sect4}
In this section, we will give some numerical examples to show the performance of our algorithm.
The parameters in all our experiments are chosen as follows: the wavelength $\lambda$ = 193nm, NA = 0.85, pixel size = 5 nm, and the size of the point spread function $H$ is $100\times 100$. We choose $tr=0.3$ in the threshold function \eqref{threshd}. The algorithm are implemented with Matlab and all the numerical experiments are done on a laptop with 1.8GHz CPU.

The edge placement error(EPE) is defined as a picture:
\begin{equation*}
	EPE=|I(U)-I|,
\end{equation*}
which shows that the error occurs at the edge of the image and is an important property to evaluate the optimality of inversion. Then we can define the error
\begin{equation*}
	error=||EPE||_2
\end{equation*}
to measure the distance between the output image with the target image, here $\|\cdot\|_2$ represents the $l^2 $ norm for vectorized image as mentioned in Section \ref{sect3}.  We give numerical examples at the best focus $D$ = 0nm and with defocus $D$ = 50nm, respectively in \eqref{defocus}. When the input is the target pattern, the output image is blurred.  After optimization, the synthesized mask will have some small and isolated assist features then the optimal mask achieves satisfactory output patterns where the corners of the image are well preserved.

\subsection{Choosing parameters}
In this example, we will discuss the influence of the parameters $\rho, \beta_1, \beta_2$ in \eqref{lag} and $\gamma$ in \eqref{breg}.  Figure (\ref{ex0}) shows the target image in this example. The size of the image is $144\times144$ and the blue part is valued 0, and the yellow represents value 1.
\begin{figure}[htbp]
	\centering
	\includegraphics[scale=0.4]{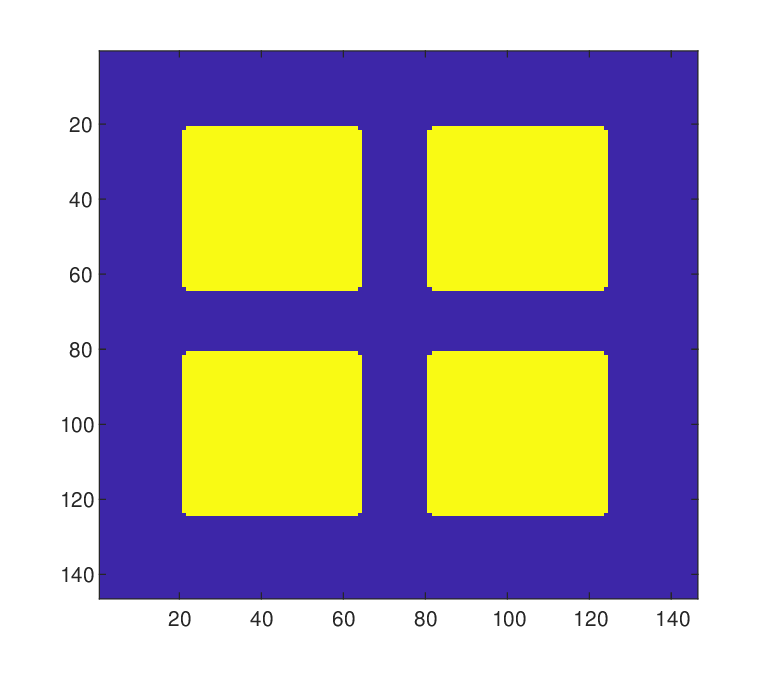}
	\caption{The target pattern}
	\label{ex0}
\end{figure}
Figure (\ref{rho}) shows error decreasing in the optimization process under different $\rho$. We can see that with the increase of $\rho$, the decline is more consistent, but more gently. So $\rho=10$ is a good choice. Figure (\ref{gamma}) gives the error decreasing for different $\gamma$. We can see that $\gamma=30$ is a suitable choice. Then in the followings, we fix $\rho=10$ and $\gamma=30$. Figure (\ref{beta1}) shows the optimized masks with different $\beta_1$. Figure (\ref{beta2}) shows the optimized images with different $\beta_2$. These numerical examples suggest us to choose $\beta_1=0.01,\beta_2=0.015$ in the following examples.
Here we remark that the choices are quite heuristic because it is quite hard to find the optimal parameters.
\begin{figure}[htbp]
	\centering
	\includegraphics[scale=0.6]{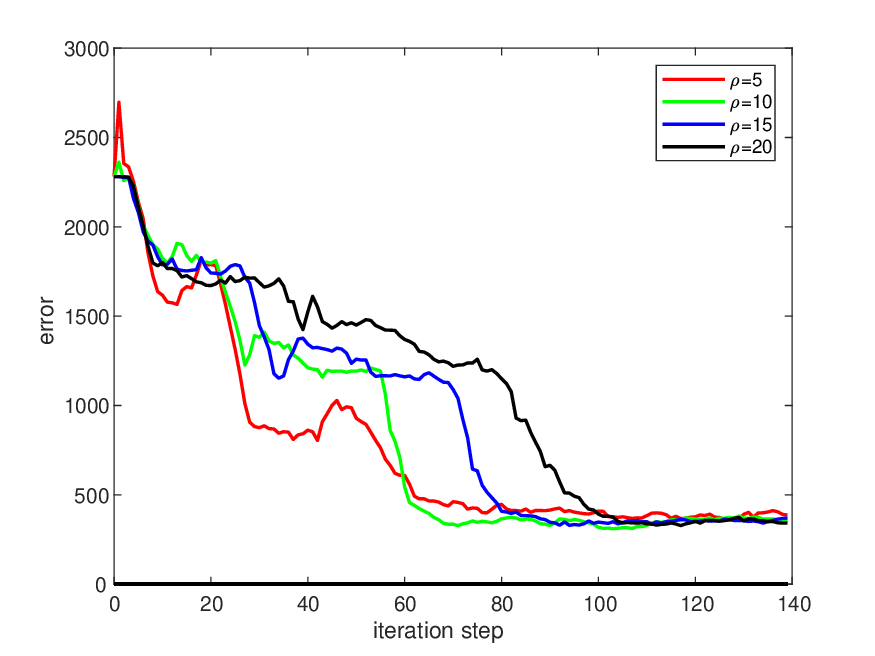}
	\caption{Error decreasing with respect to $\rho$, with $\gamma=50, \beta_1=0.005,\beta_2=0.01$}
	\label{rho}
\end{figure}
\begin{figure}[htbp]
	\centering
	\includegraphics[scale=0.6]{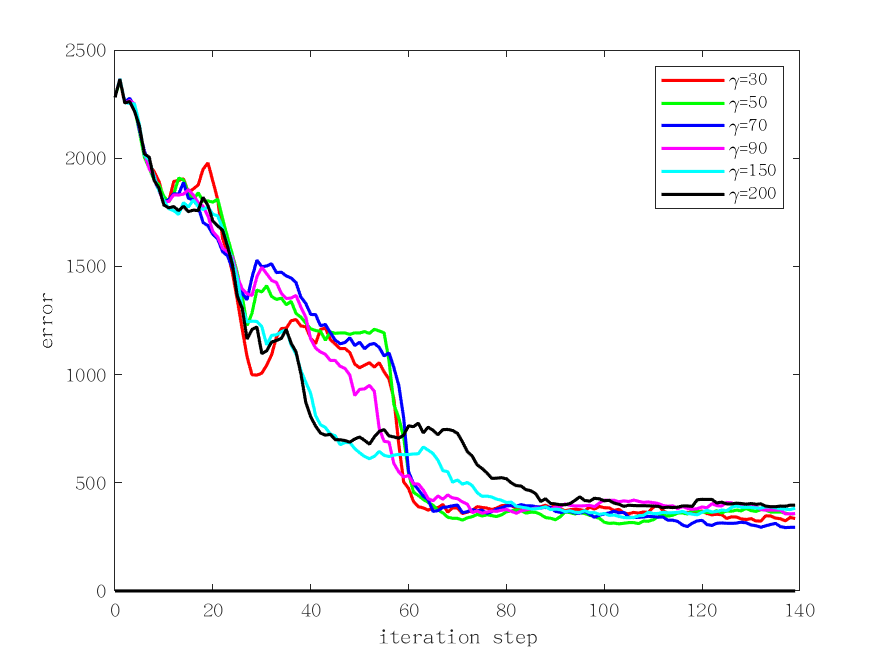}
	\caption{A heuristic choice of $\gamma$, with fixed $\rho=10, \beta_1=0.005,\beta_2=0.01$.}
	\label{gamma}
\end{figure}
\begin{figure}[htbp]
	\centering
	\includegraphics[scale=0.6]{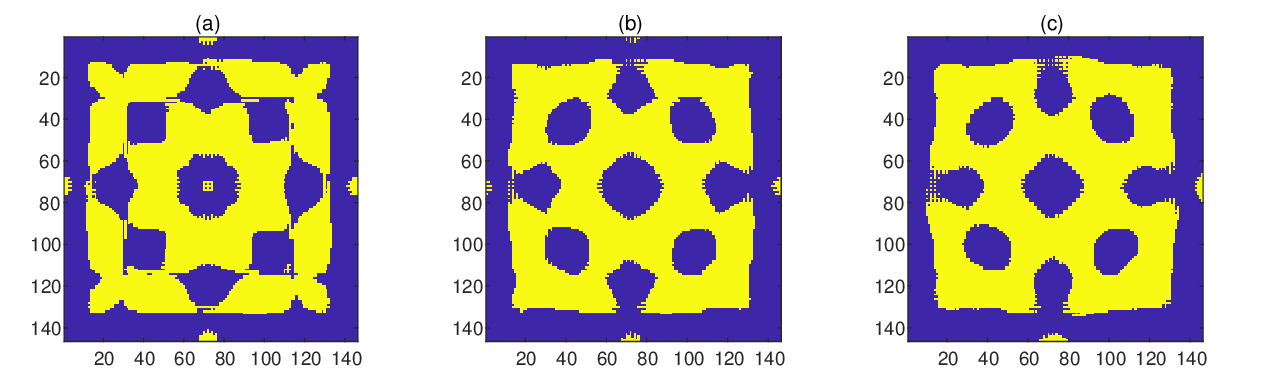}
	\caption{The optimal mask with different $\beta_1$. From left to right $\beta_1=0.005,0.01,0.015$.}
	\label{beta1}
\end{figure}
\begin{figure}[htbp]
	\centering
	\includegraphics[scale=0.6]{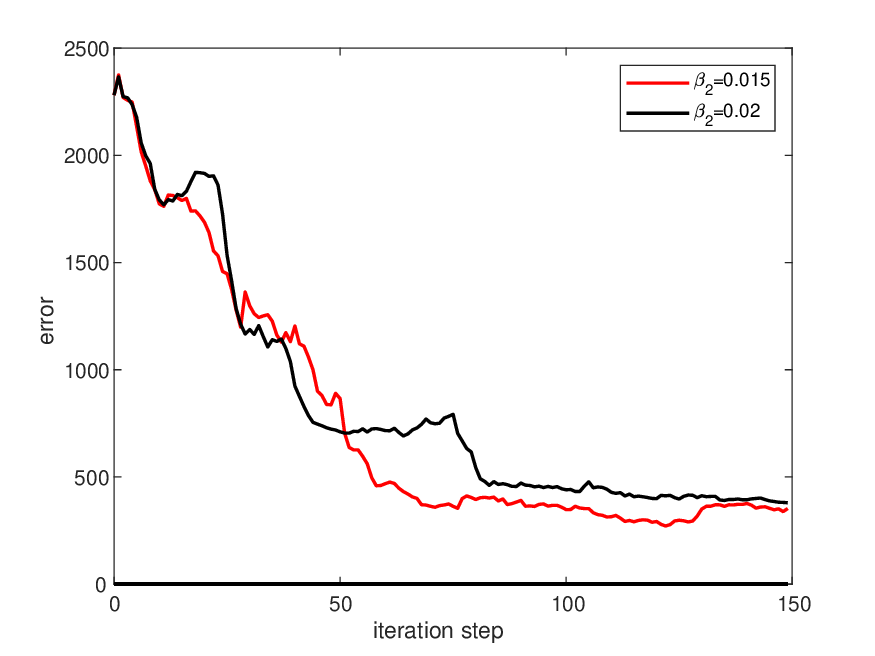}
	\caption{The output pattern with different $\beta_2$.}
	\label{beta2}
\end{figure}

\subsection{Numerical experiments}
Figures (\ref{ex1})–(\ref{ex3}) show the comparative results at
the best focus $D = 0nm$ and under defocus $D = 50nm$. 

Figure (\ref{ex1})(a) shows the target pattern consisting of 10 squares. Figure (\ref{ex1})(b) display the output patterns with target pattern as input. Figure (\ref{ex1})(c) is the edge placement error(EPE). Compared with the desired target pattern, we observe that the error is visually noticeable, especially for the regions located at the edge of each square. Figure (\ref{ex1})(d) shows the optimal mask generated by our proposed method at the best focus, Figure (\ref{ex1})(d) give the output pattern on wafer and Figure (\ref{ex1})(f) is the corresponding EPE. The last two rows displays the corresponding results under defocus $D=50nm$. As can be seen, some assist features near the squares are automatically generated in the minimization process, which would help to improve the overall pattern fidelity. In addition, the corner-rounding effect at the end of each square is reduced compared to initial pattern, and the optimal mask gives the satisfactory contours in the output patterns. 
\begin{figure}[ht]
	\centering
	\includegraphics[scale=0.6]{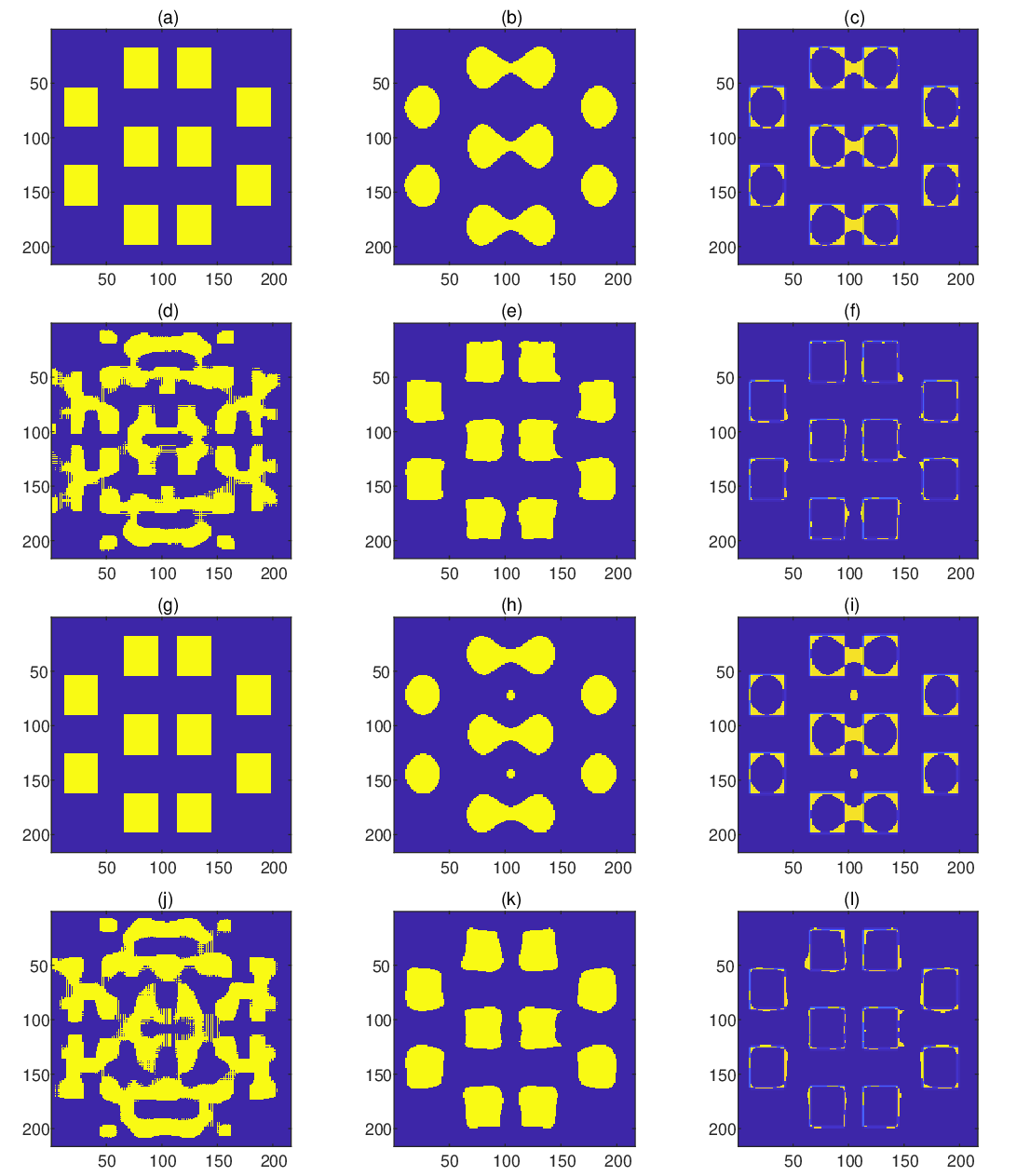}
	\caption{Left column: mask, Middle column: output pattern, Right column: EPE.}
	\label{ex1}
\end{figure}

Figure (\ref{ex2}) shows the example in which the target pattern contains long strips. From the first and third rows, we can find that the corners are round for the direct imaging with target pattern as input. By optimization with our proposed algorithm, the corners become sharp and the EPEs are significantly decreased, which can be found from the second and forth rows.
\begin{figure}[ht]
	\centering
	\includegraphics[scale=0.6]{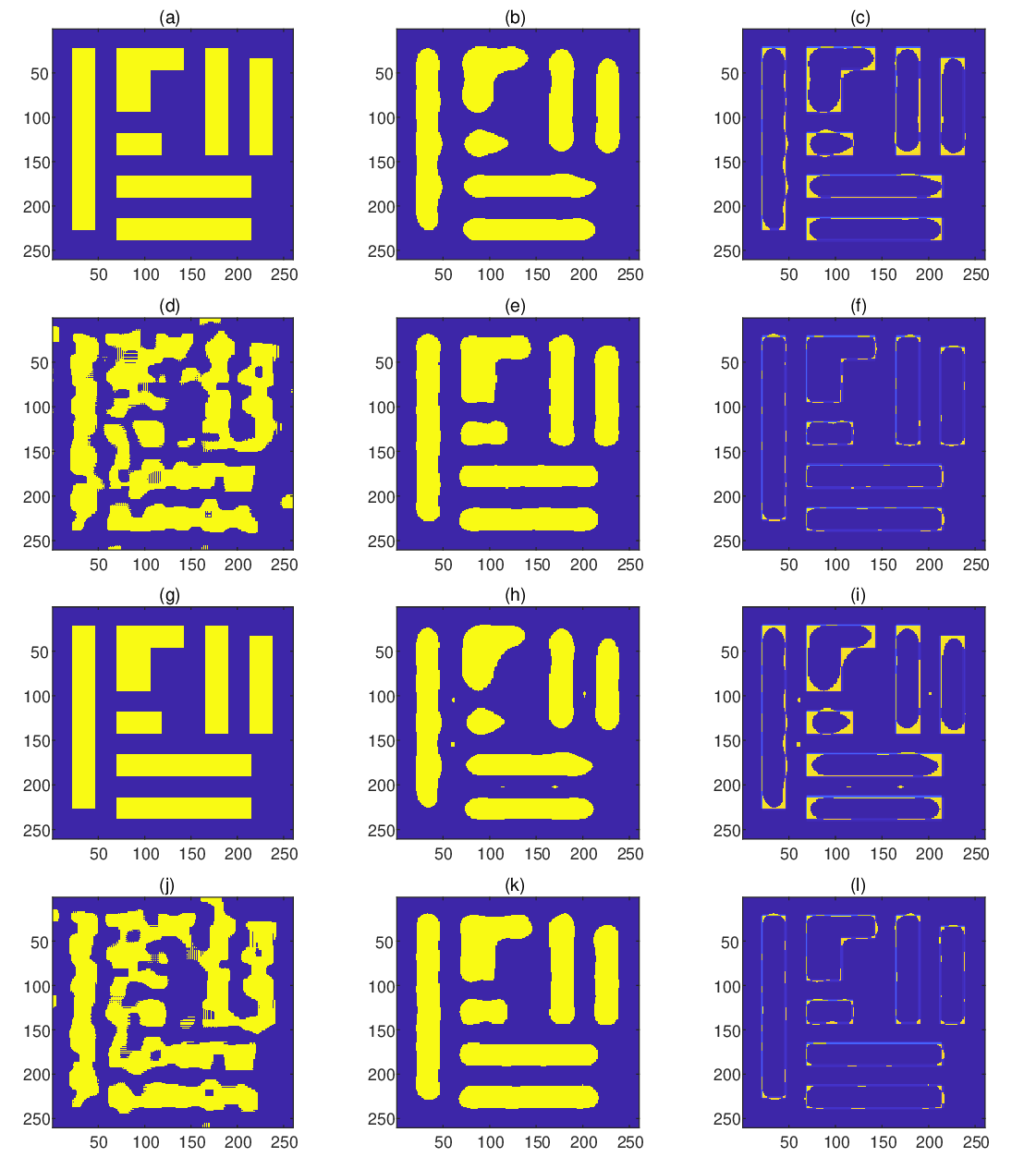}
	\caption{Left column: mask, Middle column: output pattern, Right column: EPE.}
	\label{ex2}
\end{figure}

Figure (\ref{ex3}) shows the example in which the target pattern contains small squares and long strips. For this kind of mixed pattern, the algorithm still works well. From the second and forth rows, the small details can be imaged with optimized mask patterns, even for imaging system with defocus.  
\begin{figure}[ht]
	\centering
	\includegraphics[scale=0.6]{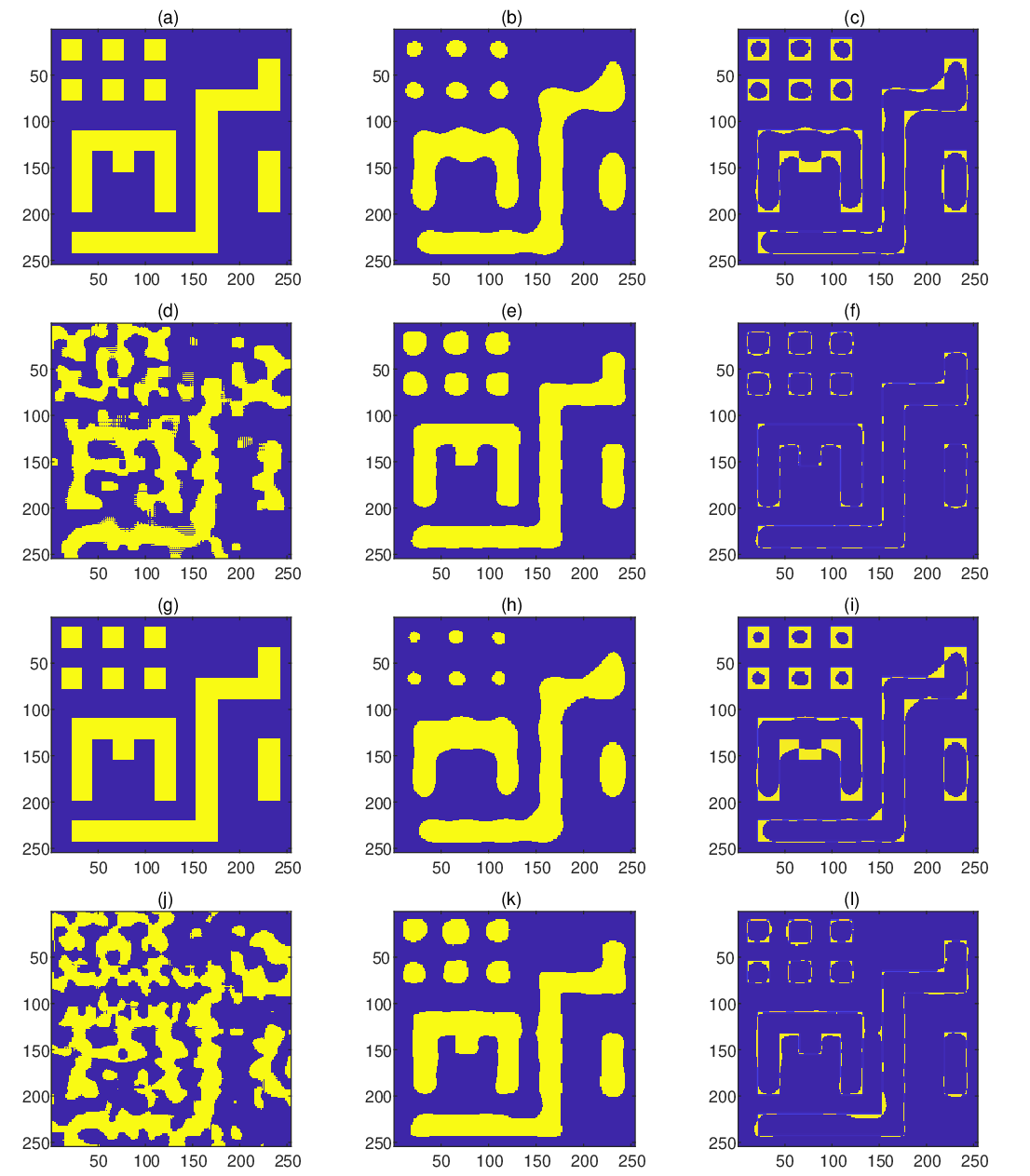}
	\caption{Left column: mask, Middle column: output pattern, Right column: EPE.}
	\label{ex3}
\end{figure}

\section{Conclusion}\label{sect5}
We have developed an ADMM method to solve the inverse lithography problem. The TV regularization is introduced to deal with ill-posedness of the inverse problem and sharp the optimized mask pattern. To deal with the TV regularization term, we propose the ADMM framework. To avoid the complicated computation of gradient of sigmoid function in gradient descent method, we use the threshold function in the ADMM subproblem directly instead of sigmoid. With this replacement, we find that the optimal solution to the subproblem is explicit. Furthermore, we prove the convergence of our proposed ADMM algorithm.  The parameters of the algorithm is chosen in a heuristic way. The numerical experiments illustrate the efficiency of the proposed algorithm.

\section*{Acknowledgments}
This research is supported partly by National Key  R\&D Program of China 2019YFA0709600, 2019YFA0709602 and by NSFC under the grant 11871300.

\bibliographystyle{unsrt}
\bibliography{ref}

\end{document}